\documentclass{amsart}
\usepackage[margin=1in]{geometry}
\usepackage{amscd}
\usepackage{amssymb}
\usepackage{amsmath}
\usepackage{amsthm}
\usepackage{enumerate}
\usepackage{tikz}
\usepackage{wrapfig, framed, caption}
\usepackage{float}
\usetikzlibrary{arrows}
\usepackage[font=small,labelfont=bf]{caption}
\usepackage{xcolor}
\usepackage{mathtools}
\usepackage[colorlinks=true, linkcolor=blue, citecolor=blue,
pagebackref=true]{hyperref}
\usepackage{fouriernc}
\usepackage{ulem}

\usepackage{comment}

\newtheorem{theorem}{Theorem}[section]
\newtheorem*{theorem*}{Theorem}

\newtheorem*{theoremY*}{Theorem Y}

\newtheorem*{theoremAB*}{Theorem AB}

\newtheorem*{linearformsmtp*}{Mass transference principle for linear forms}
\newtheorem{corollary}{Corollary}[section]
\newtheorem*{corollary*}{Corollary}

\newtheorem{lemma}{Lemma}[section]

\newtheorem*{claim*}{Claim}

\newtheorem{question}{Question}
\theoremstyle{definition}

\theoremstyle{remark}
\newtheorem{remark}{Remark}
\newtheorem*{remark*}{Remark}

\newcommand{\bp}{\mathbf{p}}

\renewcommand{\Bbb}[1]{\mathbb{#1}}

\newcommand{\N}{{\Bbb N}}         

\newcommand{\R}{{\Bbb R}}        

\newcommand{\Z}{{\Bbb Z}}         

\newcommand{\cB}{{\mathcal B}}

\newcommand{\cH}{{\mathcal H}}

\newcommand{\cK}{{\mathcal K}}

\newcommand{\cU}{{\mathcal U}}

\newcommand{\x}{\mathbf{x}}

\newcommand{\br}{\mathbf{r}}

\DeclareMathOperator{\dimh}{\dim_H}

\newcommand{\bx}{\mathbf{x}}
\newcommand{\by}{\mathbf{y}}

\newcommand{\bgam}{\mathbf{\gamma}}
\newcommand{\bfa}{\mathbf{a}}
\newcommand{\bft}{\mathbf{t}}

\title{A note on limsup sets of annuli}

\author[M. Hussain]{Mumtaz Hussain}
\address{Mumtaz Hussain,  Department of Mathematical and Physical Sciences,  La Trobe University, Bendigo, Australia. }
\email{m.hussain@latrobe.edu.au}

\author[B. Ward]{ Benjamin Ward}
\address{Benjamin Ward,  University of York, Heslington, York, YO10 5DD, United Kingdom. }
\email{benjamin.ward@york.ac.uk, ward.ben1994@gmail.com}

\begin{document}

\begin{abstract}
   We consider the set of points in infinitely many max-norm annuli centred at rational points in $\R^{n}$. We give Jarn\'ik-Besicovitch type theorems for this set in terms of Hausdorff dimension. Interestingly, we find that if the outer radii are decreasing sufficiently slowly, dependent only on the dimension $n$, and the thickness of the annuli is decreasing rapidly then the dimension of the set tends towards $n-1$. We also consider various other forms of annuli including rectangular annuli and quasi-annuli described by the difference between balls of two different norms. Our results are deduced through a novel combination of a version of Cassel's Scaling Lemma and a generalisation of the Mass Transference Principle, namely the Mass transference principle from rectangles to rectangles due to Wang and Wu (Math. Ann. 2021).
\end{abstract}
\maketitle

\section{Introduction}
A fundamental set in Diophantine approximation is the set of $\psi$-approximable points. Throughout let $\psi:\N\to \R_{+}$ be a monotonically decreasing function such that $\psi(q)\to 0$ as $q \to \infty$. Define the set of $\psi$-approximable points to be
\begin{equation*}
    W_{n}(\psi):=\left\{ \bx \in [0,1]^{n}: \left\|\bx-\tfrac{\bp}{q}\right\|<\tfrac{\psi(q)}{q} \quad \text{ for i.m. } \, (\bp,q)\in\Z^{n}\times \N \right\}\, ,
\end{equation*}
where $\|\cdot\|$ denotes the maximum norm on $\R^{n}$ and `i.m.' denotes `infinitely many'. It is well known that $W_{n}(\psi)$ can be written as a $\limsup$ set of balls centred at rational points $\tfrac{\bp}{q}$ and with radii $\tfrac{\psi(q)}{q}$. Much is known on the set $W_{n}(\psi)$, for example Jarn\'ik \cite{J28} and Besicovitch \cite{B34} independently proved that, for $\psi(q)=q^{-\tau_{\psi}}$ with $\tau \geq \tfrac{1}{n}$,
\begin{equation*}
    \dimh W_{n}(\psi) = \frac{n+1}{1+\tau_{\psi}}\, .
\end{equation*}
Here $\dimh$ denotes the Hausdorff dimension, see \cite{F14} for the definition and properties of Hausdorff dimension and measure. We consider the following variation of $W_{n}(\psi)$. Let $\phi:\N\to [0,1]$ and define the set of points
\begin{equation*}
    W_{n}(\psi,\phi):=\left\{ \bx \in [0,1]^{n}: (1-\phi(q))\frac{\psi(q)}{q} < \left\|\bx -\frac{\bp}{q}\right\|<\frac{\psi(q)}{q} \, \text{ for i.m. } \, (\bp,q)\in\Z^{n}\times \N \right\}\, .
\end{equation*}
 It is straightforward to see that $W_{n}(\psi,\phi)\subset W_{n}(\psi)$. In parallel to the set of $\psi$-approximable points, it can be seen that $W_{n}(\psi,\phi)$ can be written as a $\limsup$ set of annuli, or $n$-dimensional hyperspherical shells. Generally, we refer to these types of sets as quasi-annuli. Denoting by $A(\bx;r_{1},r_{2}):=B(\bx,r_{1})\backslash B(\bx,r_{2})$ the annulus with centre $\bx\in\R^{n}$ and inner and outer radii $r_{2},r_{1}$ respectively, we may write
\begin{equation*}
    W_{n}(\psi,\phi)=\limsup_{q\to \infty} \bigcup_{0\leq \|\bp\|\leq q}A\left(\frac{\bp}{q};\frac{\psi(q)}{q},(1-\phi(q))\frac{\psi(q)}{q} \right).
\end{equation*}
In this note, we prove the following Jarn\'ik-Besicovitch type theorem in the setting of approximation by annuli.
\begin{theorem} \label{dim n>1}
    Take $n\in \N$ and let $\psi(q)=q^{-\tau_{\psi}}$ and $\phi(q)=q^{-\tau_{\phi}}$ for $\tau_{\psi},\tau_{\phi}\in\R_{+}$ with $\tau_{\psi}\geq \frac{1}{n}$. Then
    \begin{equation*}
        \dimh W_{n}(\psi,\phi)=\min\left\{\frac{n+1}{1+\tau_{\psi}},\frac{n+1+(n-1)\tau_{\phi}}{1+\tau_{\psi}+\tau_{\phi}} \right\}\, .
    \end{equation*}
\end{theorem}
The set $W_{n}(\psi,\phi)$ can be related to sets of exact approximation order. These have recently gained significant interest, in chronogloical order see \cite{BDV01,Bu03,Bu08,BM11,Fraser23,Fregoli2023, Schl23, KLWZ,BS24}. Define
\begin{equation*}
    Exact_{n}(\psi):=\left\{ \bx \in [0,1]^{n}: \begin{cases}
        \|\bx-\tfrac{\bp}{q}\|<\frac{\psi(q)}{q} \quad \text{ for i.m. } \, \, q\in\N \\
        \text{For each } \, 0<c<1, \, \, \|\bx-\tfrac{\bp}{q}\|>c\frac{\psi(q)}{q} \quad \text{for all sufficiently large} \, \, q \in \N\, .
    \end{cases}
    \right\}.
\end{equation*}
Bandi and De Saxce \cite{BS24} recently obtained a Jarn\'ik-Besicovitch type statement in this setting. Namely, they showed that $\dimh Exact_{n}(\psi)=\frac{n+1}{1+\tau_{\psi}}$, where $\psi$ is taken as in Theorem~\ref{dim n>1} and $\tau_{\psi}>\tfrac{1}{n}$. Note the case of $n=1$ had been proven prior in \cite{BM11}. If $\bx \in Exact_{n}(\psi)$ it can be seen that there exists a sequence $(c_{q})_{q\geq 1}$ with $c_{q}\to 1$ as $q\to \infty$ such that
\begin{equation*}
    \begin{cases}
        \|\bx-\tfrac{\bp}{q}\|<\frac{\psi(q)}{q} \quad \text{ for i.m. } \, \, q\in\N \\
         \|\bx-\tfrac{\bp}{q}\|>c_{q}\frac{\psi(q)}{q} \quad \text{ for all} \, \, q \in \N\, .
    \end{cases}
\end{equation*}
A natural question to ask on the set $Exact_{n}(\psi)$ is if one can be more specific on the rate at which the sequence $(c_{q})_{q\geq 1}$ tends to one. To study this take a function $f:\N\to [0,1]$ with $f(q)\to 1$ as $q\to \infty$, and define the set of points of $(f,\psi)$-Exact approximation order by
\begin{equation*}
    Exact_{n}(f,\psi):=\left\{ \bx \in [0,1]^{n}: \begin{cases}
        \|\bx-\tfrac{\bp}{q}\|<\frac{\psi(q)}{q} \quad \text{ for i.m. } \, \, q\in\N \\
        \|\bx-\tfrac{\bp}{q}\|>f(q)\frac{\psi(q)}{q} \quad \text{ for all sufficiently large } \, \, q \in \N\, .
    \end{cases}
    \right\}\, .
\end{equation*}
Clearly $Exact_{n}(f,\psi)\subseteq Exact_{n}(\psi)$, and so we have the trivial upper bound 
\begin{equation*}
\dimh Exact_{n}(f,\psi)\leq \dimh Exact_{n}(\psi)\, .
\end{equation*}
Using the set of study in this note it can also be seen that
\begin{equation*}
    Exact_{n}( q\mapsto (1-\phi(q)),\psi)\subseteq W_{n}(\psi,\phi)\, .
\end{equation*} \par 
Note the following curious property on $\dimh W_{n}(\psi,\phi)$: for $n\geq 2$ it can be verified that 
\begin{equation*}
\tau_{\psi} \leq \tfrac{2}{n-1} \quad \text{ if and only if } \quad \frac{n+1}{1+\tau_{\psi}}\geq \frac{n+1+(n-1)\tau_{\phi}}{1+\tau_{\psi}+\tau_{\phi}}\, .
\end{equation*}
Thus, if $\tau_{\psi} \leq \tfrac{2}{n-1}$ and $\tau_{\phi}\to \infty$, then
\begin{equation*} 
    \dimh W_{n}(\psi,\phi)=n-1\, .
\end{equation*}
Whereas, if $\tau_{\psi} \geq \tfrac{2}{n-1}$, then regardless of our choice of $\tau_{\phi}$, we have that
\begin{equation*}
    \dimh W_{n}(\psi,\phi)=\frac{n+1}{1+\tau_{\psi}}\, .
\end{equation*}
This phenomenon is also seen in the case of weighted Diophantine approximation, see the comments proceeding \cite[Theorem 1]{R98}, and Remark~\ref{rectangles remark} for more details. \par 
 Combining our two observations above, Theorem~\ref{dim n>1} tells us that, for $n\geq 2$, if $\tau_{\psi}<\tfrac{2}{n-1}$ and $f(q)=1-q^{-\varepsilon}$ for some $\varepsilon>0$ then 
\begin{equation*}
\dimh Exact_{n}(f,\psi)\leq \frac{n+1+(n-1)\varepsilon}{1+\tau_{\psi}+\varepsilon} <\dimh Exact_{n}(\psi)\, .
\end{equation*}
Conversely, if $\tau_{\psi}>\tfrac{2}{n-1}$, then it may be possible that $\dimh Exact(f,\psi)=\dimh Exact(\psi)$ regardless of the rate at which $f$ tends to $1$. \par 

The rest of the note is as follows. Firstly, we generalise the above theorem in several instances. Precisely we consider various forms of quasi-annuli, including annuli formed from the difference of balls described by different norms (Theorem~\ref{eucildean norm}), and rectangular annuli (Theorem~\ref{dim >2 generalised}). Many of these results stem from a variation of inhomogeneous Diophantine approximation, which we describe in \S~\ref{perturbed}. In \S~\ref{prelim} we recall auxiliary results required in the proofs of our main theorems. Here, we use a variation of Cassel's Scaling Lemma \cite{C50} and a generalisation of the Mass Transference Principle \cite{BV06} to deduce a ``shifted mass transference principle'', Theorem~\ref{new} and Theorem~\ref{shifted mtprr}. These theorems are crucial in proving our main results, which are done in \S~\ref{prelim}-\ref{main}.

\subsection{Annuli defined by different norms}
We can consider the annuli described by taking the difference between balls described by different norms. Let $\|\cdot\|_{\rho}$ denotes the $\rho$-norm on $\R^{n}$, $\|\bx\|_{\rho}=\left(\sum_{i=1}^{n}x_{i}^{\rho}\right)^{\tfrac{1}{\rho}}$ and define the set
\begin{equation*}
    W_{n}(\psi, \|\cdot\|, \|\cdot\|_{\rho}):= \left\{ \bx \in [0,1]^{n}: \begin{cases}
        \|\bx-\tfrac{\bp}{q}\| < \tfrac{\psi(q)}{q} , \\
        \|\bx-\tfrac{\bp}{q}\|_{\rho}>\tfrac{\psi(q)}{q},
    \end{cases} \quad \text{for i.m.} \, q \in \N \right\}\, .
\end{equation*}
 For $\rho_1,\rho_2 \in \R_{+}\cup\{\infty\}$, $\bx \in [0,1]^{n}$ and $r>0$ let 
 \begin{equation*}
     A^{\rho_1}_{\rho_{2}}\left(\bx;r\right):= B_{\rho_{1}}(\bx,r)\backslash B_{\rho_{2}}(\bx,r) \, .
 \end{equation*}
 Then we may write
 \begin{equation*}
      W_{n}(\psi, \|\cdot\|, \|\cdot\|_{\rho})=\limsup_{q\to \infty} \bigcup_{0\leq \|\bp\|\leq q}A^{\infty}_{\rho}\left(\tfrac{\bp}{q};\tfrac{\psi(q)}{q}\right)\, .
 \end{equation*}
We have the following result.
\begin{theorem} \label{eucildean norm}
Let $\psi(q)=q^{-\tau_{\psi}}$ for $\tau_{\psi}\geq \tfrac{1}{n}$. Then, for any $\rho$-norm,
\begin{equation*}
    \dimh W_{n}(\psi,\|\cdot\|,\|\cdot\|_{\rho}) = \frac{n+1}{1+\tau_{\psi}}\, .
\end{equation*}
\end{theorem}
 \begin{remark} \rm
 What will become clear in the proof of Theorem~\ref{eucildean norm} is that we can take any norm that allows us to construct a ball of radius comparable, up to a multiplicative constant, with $\frac{\psi(q)}{q}$ within the resulting quasi-annulus.
\end{remark}
The following is a natural open problem akin to the sets of exact approximation order as defined above.\par
    \begin{question}
    Define
     \begin{equation*}
         Exact^{*}_{n}(\psi,\|\cdot\|,\|\cdot\|_{\rho}):= \left\{ \bx \in [0,1]^{n}: \begin{cases}
        \|\bx-\tfrac{\bp}{q}\| < \tfrac{\psi(q)}{q} , \quad \text{ for i.m. } q \in \N \\
        \|\bx-\tfrac{\bp}{q}\|_{\rho}>\tfrac{\psi(q)}{q}, \quad \text{for all sufficiently large} \, q \in \N 
        \end{cases} \right\}\, .
     \end{equation*}
     What can be said on the size of $Exact^{*}_{n}(\psi,\|\cdot\|,\|\cdot\|_{\rho})$? 
     \end{question}\par 
     One can use Khintchine's Theorem and the equivalence of norms to deduce that $Exact^{*}_{n}(\psi,\|\cdot\|,\|\cdot\|_{\rho})$ is a Lebesgue nullset. Notice that it is not true that $Exact_{n}(\psi) \subseteq Exact_{n}^{*}(\psi,\|\cdot\|,\|\cdot\|_{\rho})$, so the Hausdorff dimension result does not follow immediately from that of \cite{BS24}. However, it may be possible using the techniques of \cite{BS24}, or \cite{KLWZ}, to obtain the Hausdorff dimension of $Exact_{n}^{*}(\psi,\|\cdot\|,\|\cdot\|_{\rho})$ provided one can impose some additional ``directional dependence'' on the constructed Cantor sets in the respective papers. It would seem reasonable to conjecture that $$\dimh Exact_{n}^{*}(\psi,\|\cdot\|,\|\cdot\|_{\rho})=\tfrac{n+1}{1+\tau_{\psi}}.$$



\subsection{ Approximation by rectangular annuli } \label{rectangles}

Weighted Diophantine approximation generalises the set $W_{n}(\psi)$ by considering $\limsup$ sets of rectangles, rather than $\limsup$ set of balls. These sets have been well-studied, in particular, an analogue of the Jarn\'ik-Besicovitch theorem in the weighted case was proven by Rynne \cite{R98}. Similar to our setup in the introduction, we can also consider a weighted version of Diophantine approximation by annuli. \par 
Let $\Psi=(\psi_{1},\dots,\psi_{n})$ and $\Phi=(\phi_{1},\dots,\phi_{n})$ be $n$-tuples of monotonic decreasing functions $\psi_{i}:\N\to \R_{+}$ with each $\psi_{i}(q)\to 0$ as $q\to \infty$ and each $\phi_{i}:\N\to (0,1]$ for $i=1,\dots,n$. Define
\begin{equation*}
    W_{n}(\Psi,\Phi):=\left\{ \bx \in [0,1]^{n} : \begin{cases}
        \left|x_{i}-\frac{p_{i}}{q}\right|<\frac{\psi_{i}(q)}{q} \qquad \qquad \text{ for each } i=1,\dots, n \, , \\
         \left|x_{j}-\frac{p_{j}}{q}\right| > \frac{(1-\phi_{j}(q))\psi_{j}(q)}{q}\,\, \, \text{ for some } j \in \{1,\dots,n\}\, ,
        \end{cases}
         \text{ for i.m. } (\bp,q) \in \Z^{n}\times\N \right\}\, .
\end{equation*}
By defining the rectangular annuli about rational points $\tfrac{\bp}{q}$ as
\begin{equation*}
    A_{\bp,q}(\Psi,\Phi):= \left(\prod_{i=1}^{n}B\left(\tfrac{p_{i}}{q},\tfrac{\psi_{i}(q)}{q}\right)\right) \backslash \left(\prod_{i=1}^{n}B\left(\tfrac{p_{i}}{q},(1-\phi_{i}(q))\tfrac{\psi_{i}(q)}{q}\right)\right)\, ,
\end{equation*}
we may write
\begin{equation*}
    W_{n}(\Psi,\Phi)=\limsup_{q\to \infty}\bigcup_{0\leq \|\bp\|\leq q} A_{\bp,q}(\Psi,\Phi)\, .
\end{equation*}
We prove the following result.
\begin{theorem}\label{dim >2 generalised}
    Let each $\psi_{i}(q)=q^{-\tau_{\psi_i}}$ and each $\phi_{i}(q)=q^{-\tau_{\phi_i}}$ with $\tau_{\psi_i},\tau_{\phi_i} \in \R_{+}$. Additionally, assume that
    \begin{equation*}
        \sum_{i=1}^{n}\tau_{\psi_i}\geq 1\, .
    \end{equation*}
    Then
    \begin{equation*}
        \dimh W_{n}(\Psi,\Phi) = \max_{j=1,\dots ,n}\,\, \min_{k(j)=1,\dots,n} \left\{ \frac{n+1+\sum_{i\neq j : \tau_{\psi_i}<\tau_{k(j)}}(\tau_{k(j)}-\tau_{\psi_i})+\max\{0,\tau_{k(j)}-\tau_{\psi_j}-\tau_{\phi_j}\}}{1+\tau_{k(j)}} \right\}\, , 
    \end{equation*}
    where
    \begin{equation*}
        \tau_{k(j)}=\begin{cases}
            \tau_{\psi_j}+\tau_{\phi_j} \quad \text{ if } k(j)=j\, ,\\
            \tau_{\psi_{k(j)}}  \quad \qquad \text{ otherwise.}
        \end{cases}
    \end{equation*}
\end{theorem}

\begin{remark}\rm \label{rectangles remark}
    Observe that if $\tau_{\psi_i}=\tau_{\psi_j}$ and $\tau_{\phi_i}=\tau_{\phi_j}$ for each $i,j \in \{1,\dots,n\}$ then we reduce to Theorem~\ref{dim n>1}. Thus by proving Theorem~\ref{dim >2 generalised} we readily prove Theorem~\ref{dim n>1}. As previously mentioned, even in the case when $\tau_{\psi_i}=\tau_{\psi_j}$ and $\tau_{\phi_i}=\tau_{\phi_j}$ for each $i,j \in \{1,\dots,n\}$, we have this phenomenon where the rate of decay of the inner radius can be ignored if the decay rate of the outer radii is sufficiently fast, that is $\tau_{\psi}\geq \tfrac{2}{n-1}$. This same curiosity appears in the case of weighted Diophantine approximation. In the case of weighted approximation one can think of the result as arguing that for $\tau_{\psi}$ sufficiently large one does not lose any accuracy in the dimension result by considering the dimension of the approximation set intersected with a coordinate hyperplane through the origin, hence why at the bound $\tau_{\psi}=\tfrac{2}{n-1}$ one obtains the dimension $n-1$. This idea works since there are infinitely many rational points on such a hyperplane. This does not appear to be the case in the setting presented in this paper, yet we obtain the same result.
\end{remark}

%
%
%
%

\section{Perturbed Diophantine approximation} \label{perturbed}
In this section, we describe a generalised notion of inhomogeneous approximation that is used to prove Theorem~\ref{eucildean norm}. The setup presented here may also be of independent interest. \par 
Let $\psi:\N\to \R_{+}$ be a function with $\psi(q)\to 0$ as $q\to \infty$ and let $\bgam:\Z^{n+1}\to[-\tfrac{1}{2},\tfrac{1}{2}]^{n}$ with $\|\bgam(\bp,q)\|\to 0$ as $q\to \infty$. Define the set of $\bgam$-perturbed $\psi$-approximable points as the set
\begin{equation*}
W^{\bgam}_{n}(\psi):=\left\{ \bx \in [0,1]^{n}: \|q\bx-\bp+\bgam(\bp,q)\|<\psi(q) \quad \text{ for i.m. } (\bp,q) \in \Z^{n}\times \N \right\}.
\end{equation*}
Here $W^{\bgam}_{n}(\psi)$ can be thought of as a $\limsup$ set of balls similar to $W_{n}(\psi)$, but with centres perturbed by $\tfrac{\bgam(\bp,q)}{q}$. We prove the following statement on the Hausdorff $f$-measure, denoted $\cH^{f}$, of $W^{\bgam}_{n}(\psi)$.
\begin{theorem} \label{perturbed statement}
Let $\psi:\N\to\R_{+}$ and $\bgam:\Z^{n+1}\to [-\tfrac{1}{2},\tfrac{1}{2}]^{n}$ be as above and let $f:\R_{+}\to \R_{+}$ be a dimension function such that $r^{-n}f(r)$ is monotonic. Suppose that
\begin{equation}\label{kick size}
\underset{\psi(q)\neq 0}{\limsup_{q \to \infty}}\frac{\max_{1\leq \|\bp\|\leq q}\|\bgam(\bp,q)\|}{q f\left(\tfrac{\psi(q)}{q}\right)^{\tfrac{1}{n}}}<\infty\, .
\end{equation}
Then
\begin{equation*}
\cH^{f}\left(W^{\bgam}_{n}(\psi)\right)=\begin{cases}
0 \qquad \qquad \,\,\, \, &\text{\rm if } \quad\sum_{q=1}^{\infty} q^{n} f\left(\tfrac{\psi(q)}{q}\right)<\infty \, , \\
\cH^{f}([0,1]) \quad &\text{\rm if }\quad \sum_{q=1}^{\infty} q^{n} f\left(\tfrac{\psi(q)}{q}\right)=\infty \, .
\end{cases}
\end{equation*}
\end{theorem}

 For the Hausdorff dimension, we can deduce the following:
 \begin{corollary} \label{perturbed corollary}
Let $\psi(q)=q^{-\tau_{\psi}}$ for some $\tau_{\psi}\geq \tfrac{1}{n}$ and $\bgam:\Z^{n+1}\to [-\tfrac{1}{2},\tfrac{1}{2}]^{n}$. Suppose that
\begin{equation*}
    \limsup_{q\to \infty} \frac{\max_{1\leq \|\bp\|\leq q} \|\bgam(\bp,q)\|}{q^{-\tfrac{1}{n}}}<\infty\, .
\end{equation*}
Then $\dimh W^{\bgam}_{n}(\psi)=\frac{n+1}{1+\tau_{\psi}}$.
 \end{corollary}
 Theorem~\ref{perturbed statement} and Corollary~\ref{perturbed corollary} essentially tell us that we can perturb our rational points by some small amount and still obtain not only the same Hausdorff dimension but also the same Hausdorff $f$-measure as in Jarn\'ik's theorem. Note that the shift can depend on both $q$ and $\bp$. In our application of Theorem~\ref{perturbed statement} to prove Theorem~\ref{eucildean norm} our perturbation is described in terms of $q$ only. It would be interesting to find an application that utilises the full strength of Theorem~\ref{perturbed statement}. \par 
 Notice that Corollary~\ref{perturbed corollary} makes clear that the perturbation can be made independent of $\psi$. However, in both Theorem~\ref{perturbed statement} and Corollary~\ref{perturbed corollary} we require some control over the size of the perturbation, this leads to the following open question.
\begin{question}
    Can condition \eqref{kick size} be weakened?\footnote{This problem, in dimension $1$, was introduced to the second name author by Gerardo Gonzalez-Robert and Nikita Shulga, who in turn were introduced to the problem by Alexander Fish.}.
    That is, how much can we perturb the rational points $\tfrac{\bp}{q}$ while leaving the Hausdorff $f$-measure, or Hausdorff dimension, unchanged? 
    \end{question}\par 
    
    With Theorem~\ref{perturbed statement} we are now in a position to prove Theorem~\ref{eucildean norm}.
\begin{proof}[Proof of Theorem~\ref{eucildean norm}]
The upper bound Hausdorff dimension follows from a standard covering argument. \par 
For the lower bound Hausdorff dimension result consider the following subset. Note that for any $\bx\in[0,1]^{n}$, $0<r<1$ and $0<\rho<\infty$ we have that
\begin{equation*}
    B_{\infty}\left(\tfrac{\bp}{q},\tfrac{\psi(q)}{q}\right)\backslash B_{\rho}\left(\tfrac{\bp}{q},\tfrac{\psi(q)}{q}\right) \supseteq B_{\infty}\left(\tfrac{\bp}{q} + \left(\pm \tfrac{\psi(q)}{q}\tfrac{n^{1/\rho}+n^{-1/\rho}}{2},\dots, \pm \tfrac{\psi(q)}{q}\tfrac{n^{1/\rho}+n^{-1/\rho}}{2}\right), \tfrac{\psi(q)}{q}(n^{1/\rho}-n^{-1/\rho}) \right) \, ,
\end{equation*}
where the sign in each coordinate axis is chosen to ensure we remain inside $[0,1]^{n}$. Thus
\begin{equation*}
    \limsup_{q\to \infty} \bigcup_{0\leq \|\bp\|\leq q} B_{\infty}\left(\tfrac{\bp}{q} + \left(\pm \tfrac{\psi(q)}{q}\tfrac{n^{1/\rho}+n^{-1/\rho}}{2},\dots, \pm \tfrac{\psi(q)}{q}\tfrac{n^{1/\rho}+n^{-1/\rho}}{2}\right), \tfrac{\psi(q)}{q}(n^{1/\rho}-n^{-1/\rho}) \right) \subseteq W_{n}(\psi,\|\cdot\|,\|\cdot\|_{\rho})\, . 
\end{equation*}
We now apply Theorem~\ref{perturbed statement} to the above $\limsup$ set. Let
\begin{equation*}
    \bgam(\bp,q)=\left(\pm \psi(q)\tfrac{n^{1/\rho}+n^{-1/\rho}}{2},\dots, \pm \psi(q)\tfrac{n^{1/\rho}+n^{-1/\rho}}{2}\right) \, .
\end{equation*}
Pick dimension function $f(r)=r^{\tfrac{n+1}{1+\tau_{\psi}}}$. Then $r^{-n}f(r)$ is monotonic and so our conditions in Theorem~\ref{perturbed statement} on the dimension function are satisfied. By our condition that $\tau_{\psi}\geq \tfrac{1}{n}$, for sufficiently large $q\in \N$ we have that
\begin{equation*}
    \limsup_{q\to \infty}\frac{q^{-\tau_{\psi}}\tfrac{n^{1/\rho}+n^{-1/\rho}}{2}}{q\left((n^{1/\rho}-n^{-1/\rho})q^{-(1+\tau_{\psi})}\right)^{\tfrac{n+1}{1+\tau_{\psi}}\tfrac{1}{n}}}=\limsup_{q\to \infty} q^{-\tau_{\psi}+\tfrac{1}{n}}\tfrac{n^{1/\rho}+n^{-1/\rho}}{2(n^{1/\rho}-n^{-1/\rho})^{\tfrac{n+1}{1+\tau_{\psi}}\tfrac{1}{n}}} \leq \tfrac{n^{1/\rho}+n^{-1/\rho}}{2(n^{1/\rho}-n^{-1/\rho})^{\tfrac{n+1}{1+\tau_{\psi}}\tfrac{1}{n}}} < \infty \, , 
\end{equation*}
and so condition~\eqref{kick size} is satisfied. To complete the proof it is easy to check that we are in the divergence case of Theorem~\ref{perturbed statement} and so we obtain the lower bound dimension result.
\end{proof}

%
%
%
%

\section{Preliminaries}\label{prelim}

In this section we novelly combine variations of Cassel's Scaling Lemma with generalisations of Beresnevich and Velani's Mass Transference Principle to give us shifted Mass Transference Principles.\par 
Let us begin by defining the setup. For each $i\in \{1,\dots,n\}$ let $(X_{i},d_{i},\mu_{i})$ be a bounded locally compact metric space equipped with a $\delta_{i}$-Ahlfors regular probability measure $\mu_{i}$ with $\text{supp} \mu_{i}= X_{i}$, and let $(X,d,\mu)$ denote the product space with $X=\prod_{i=1}^{n}X_{i}$, $d=\max_{i}d_{i}$ and $\mu=\prod_{i} \mu_{i}$ the product measure. Note that $\mu$ is $\delta:=\sum_{i}\delta_{i}$-Ahlfors regular. Let $B^{(i)}=B^{(i)}(x,r)\subset X_{i}$ denote a ball in $X_{i}$ with center $x \in X_{i}$ and radius $r(B)=r>0$, and let $R(\bx,\br)=\prod_{i=1}^{n}B^{(i)}(x_{i},r_{i})\subset X$ denote the hyperrectangle with center $\bx \in X$ and sidelengths $2r_{i}$ in each $i$th metric space. If each $r_{i}=r^{a_{i}}$ for some $a_{i}\in \R_{+}$ let $\bfa=(a_{1},\dots,a_{n})$ and let $R(\bx,r^{\bfa})=\prod_{i=1}^{n}B^{(i)}(x_{i},r^{a_{i}})$. Lastly, if $r=r_{1}=\dots=r_{n}$ let $B(\bx,r)=\prod_{i=1}^{n}B^{(i)}(x_{i},r)$ denote the ball in $X$ with center $\x\in X$ and radius $r>0$. \par
In this note we will often take $X_{i}=[0,1]$, so $X=[0,1]^{n}$, $d_{i}(x,y)=|x-y|$, so $d(\bx,\by)=\|\bx-\by\|$ the max norm, $\mu_{i}=\lambda$ for $\lambda$ the $1$-dimensional Lebesgue measure. \par
Recall a dimension function $f:\R_{+}\to \R_{+}$ is a continuous, nondecreasing function such that $f(r)\to 0$ as $r\to 0$. We will use the following notations for manipulations of a ball $B=B(\bx,r)$ with centre $\bx \in X$ and radius $r>0$: 
\begin{itemize}
    \item For dimension function $f$, let $B^{f}=B(\bx,f(r)^{\tfrac{1}{\delta}})$ denote the ball with the same centre, but the radius blown up depending on $\delta$ and $f$.
    \item For any constant $c>0$, let $cB=B(\bx,cr)$ denote the ball with the same centre as $B$ but with radius scaled by a factor $c$.
    \item For $\by \in X$, let $(B+\by)=B(\bx+\by,r)$ denote the ball $B$ with unchanged radius and centre shifted by $\by$.
\end{itemize}
 The well-known Mass Transference Principle of Beresnevich-Velani \cite{BV06} is as follows:
\begin{theorem}[\cite{BV06}]\label{MTP}
Let $(X,d,\mu)$ be as above and let $(B_{i})_{i\geq1}$ be a sequence of balls with radius $r(B_{i})\to 0$ as $i \to \infty$. Let $f$ be a dimension function such that $r^{-\delta}f(r)$ is monotonic. Suppose that for any ball $B\subseteq X$
\begin{equation*}
\mu\left(B\cap \limsup_{i \to \infty} B_{i}^{f}\right)=\mu(B). 
\end{equation*}
Then, for any ball $B\subseteq X$
\begin{equation*}
\cH^{f}\left(B\cap \limsup_{i\to \infty} B_{i}\right)=\cH^{f}(B).
\end{equation*}
\end{theorem}

We now give the second measure-theoretic tool required. The following lemma is a generalisation of Cassel's scaling Lemma \cite{C50} for balls.

\begin{lemma}[{\cite[Lemma 5.3]{BDGW23}}] \label{CI_balls}
Let $(X, d, \mu)$ be as above and let $(B_{i})_{i \in \N}$ be a sequence of balls in $X$ with the radii $r(B_{i}) \to 0$ as $i \to \infty$. 
Let $(U_{i})_{i \in \N}$ be a sequence of $\mu$-measurable subsets of $X$ such that $U_{i} \subset B_{i}$ for each $i \in \N$ and there exists some fixed constant $c_{2}>0$ such that
\begin{equation*}
\mu(U_{i}) \geq c_{2}\,\mu(B_{i}).
\end{equation*}
Then the difference of the limsup sets
\begin{equation*}
\cB=\limsup_{i \to \infty} B_{i} \qquad \text{and} \qquad \cU=\limsup_{i \to \infty} U_{i}
\end{equation*}
is null, that is $\mu(\cB\setminus\cU)=0$.
\end{lemma}
In \cite{BDGW23} the above Lemma is relaxed somewhat. Namely rather than a Ahlfors regular measure one can consider a $\sigma$-finite Borel regular measure providing some additional conditions are met.\par


Combining these results we can prove the following modified Mass Transference Principle.
\begin{theorem} \label{new}
Let $(X,d,\mu)$ be as above and let $(B_{i})_{i\geq1}$ be a sequence of balls $B_{i}=B(x_{i},r_{i})$ with each $x_{i} \in X$ and radius $r_{i}\to 0$ as $i \to \infty$. Let $f$ be a dimension function such that $r^{-\delta}f(r)$ is monotonic. Suppose that for any ball $B\subseteq X$
\begin{equation} \label{measure}
\mu\left(B\cap \limsup_{i \to \infty} B_{i}^{f}\right)=\mu(B),
\end{equation}
and suppose $(\gamma_{i})_{i\geq1}$ is a sequence such that $x_{i} + \gamma_{i}\in X$ and
\begin{equation} \label{shift}
\limsup_{i\to \infty} \frac{d(0,\gamma_{i})}{f(r_{i})^{\tfrac{1}{\delta}}}<\infty\, .
\end{equation}
Then for any ball $B\subseteq X$
\begin{equation*}
\cH^{f}\left(B\cap \limsup_{i\to \infty} (B_{i}+\gamma_{i})\right)=\cH^{f}(B).
\end{equation*}
\end{theorem}

\begin{proof}
By \eqref{shift} and the triangle inequality, there exists some constant $k>0$ such that for all sufficiently large $i\geq1$ we have that
$(k+1)B_{i}^{f}\supseteq (B_{i}^{f}+\gamma_{i})$, and by the doubling property of $\mu$ and the fact that $x_i +\gamma_i \in X$, we have that $\mu((k+1)B_{i}^{f})\leq k' \mu(B_{i}^{f}+\gamma_{i})$. By repeatedly applying Lemma~\ref{CI_balls}, to the pairs of sequences of balls $((B_{i}^{f}+\gamma_{i}),kB_{i}^{f})_{i\geq 1}$ and then $(kB_{i}^{f},B_{i}^{f})_{i\geq1}$ we have the following chain of equalities
\begin{equation*}
\mu\left(\limsup_{i\to \infty}(B_{i}^{f}+\gamma_{i})\right)=\mu\left(\limsup_{i\to \infty} kB_{i}^{f}\right)=\mu\left(\limsup_{i\to \infty}B_{i}^{f}\right).
\end{equation*}
Hence \eqref{measure} implies that for any ball $B\subseteq X$
\begin{equation*}
\mu\left(B\cap \limsup_{i\to \infty} (B_{i}^{f}+\gamma_{i})\right)=\mu(B).
\end{equation*}
Applying Theorem~\ref{MTP} to the above equation immediately gives us our result.
\end{proof}

Armed with Theorem~\ref{new} we may now prove the theorem on perturbed approximation.

\begin{proof}[Proof of Theorem~\ref{perturbed statement}]
A straightforward covering argument provides the convergence case of Theorem~\ref{perturbed statement} via the Hausdorff-Cantelli Lemma. \par 
For the divergence case recall that Khintchine's Theorem \cite{K24} tells us that for any monotonic decreasing approximation function $\theta:\N\to \R_{+}$ with $\theta(q)\to 0$ as $q\to \infty$
\begin{equation*}
    \lambda_{n}\left(W_{n}(\theta)\right)=1\quad \text{ if } \, \sum_{q=1}^{\infty}q^{n}\left(\frac{\theta(q)}{q}\right)^{n}=\infty\, .
\end{equation*}
Taking $\tfrac{\theta(q)}{q}=f\left(\tfrac{\psi(q)}{q}\right)^{\tfrac{1}{n}}$ then since we are considering the divergence case we have that for any ball $B\subseteq [0,1]^{n}$
\begin{equation*}
\lambda_{n}\left(B\cap W_{n}\left(f\left(\tfrac{\psi(q)}{q}\right)^{\tfrac{1}{n}}\right)\right)=\lambda_{n}(B)
\end{equation*}
and so condition \eqref{measure} is satisfied. Clearly assumption \eqref{kick size} implies condition \eqref{shift} and so we may apply Theorem~\ref{new} to deduce that for any ball $B\subseteq [0,1]^{n}$
\begin{equation*}
    \cH^{f}\left(B\cap W_{\bgam}(\psi)\right)=\cH^{f}(B) \quad \text{ if } \, \sum_{q=1}^{\infty}q^{n}f\left(\tfrac{\psi(q)}{q}\right) = \infty \, 
\end{equation*}
completing the proof.
\end{proof}
We note here that Theorem~\ref{new} can be used for far more general results. In particular, one could look at perturbed approximation by reduced fractions, that is,  a Duffin-Schaeffer type result. The above argument would clearly translate by replacing Khintchine's Theorem with the theorem of Koukoulopoulos and Maynard~\cite{KM20}.

\subsection{A shifted mass transference principle for rectangles}

Analogous to the previous section we can develop a similar statement to Theorem~\ref{new} for $\limsup$ sets of hyperrectangles. For a mass Transference Principle from rectangles to rectangles, we recall the following Theorem of Wang and Wu \cite{WW19}.

\begin{theorem}[{\cite[Theorem 3.4 + Proposition 3.1]{WW19}}] \label{MTPRR}
Let $(X,d,\mu)$ be as above and let $(r_{i})_{i\geq 1}$ be a sequence of positive real numbers with $r_{i}\to 0$ as $i\to \infty$ and $(\bx_{i})_{i\geq 1}$ a countable collection of points in $X$. Let $\bfa=(a_{1},\dots,a_{n})\in~\R^{n}_{+}$ and suppose that
\begin{equation*}
    \mu\left( \limsup_{i\to \infty} R(\bx_{i},r_{i}^{\bfa}) \right)=\mu(X)\, .
\end{equation*}
Then, for any $\bft=(t_{1},\dots, t_{n})\in\R^{n}_{+}$,
\begin{equation*}
\dimh \limsup_{i\to \infty}R\left(\bx_{i},r_{i}^{\bfa+\bft}\right) \geq \min_{A_{i} \in A} \left\{ \sum\limits_{j \in \cK_{1}} \delta_{j}+ \sum\limits_{j \in \cK_{2}} \delta_{j}+ \kappa \sum\limits_{j \in \cK_{3}} \delta_{j}+(1-\kappa) \frac{\sum\limits_{j \in \cK_{3}}a_{j}\delta_{j}-\sum\limits_{j \in \cK_{2}}t_{j}\delta_{j}}{A_{i}} \right\},
\end{equation*}
where $A=\{ a_{i}+t_{i} , 1 \leq i \leq n \}$ and $\cK_{1},\cK_{2},\cK_{3}$ are a partition of $\{1, \dots, n\}$ defined as
\begin{equation*}
 \cK_{1}=\{ j:a_{j} \geq A_{i}\}, \quad \cK_{2}=\{j: a_{j}+t_{j} \leq A_{i} \} \backslash \cK_{1}, \quad \cK_{3}=\{1, \dots n\} \backslash (\cK_{1} \cup \cK_{2}).
 \end{equation*}
 \end{theorem}

We need the following generalisation, for the analogous result of Cassel's Scaling Lemma, for $\limsup$ sets of hyperrectangles. The proof of the following statement follows essentially \cite[Lemma 5.7]{BDGW23} that we include for completeness.

\begin{lemma} \label{CI_product}
Let $(X,d,\mu)$ be as above and let $(\br_{i})_{i\geq 1}$ be a sequence of $n$-dimensional positive real numbers with $\|\br_{i}\|\to 0$ as $i\to \infty$. Let $(\bx_{i})_{i\geq 1}$ be a countable collection of points in $X$. Let $(U_{i})_{i\geq 1}$ be a sequence of sets of the form $U_{i}=\prod_{j=1}^{n}U^{(j)}_{i}$ with each $U^{(j)}_{i}$ $\mu_{j}$-measurable. Suppose that $U_{i}\subseteq R(\bx_{i},\br_{i})$ and that there exists constant $c>0$ such that
\begin{equation} \label{CI_product condition}
    \mu_{j}\left(B^{(j)}(x_{i,j},r_{i,j})\right) \leq c \mu_{j}\left(U^{(j)}_{i}\right)\, \quad \text{ for each } j\in \{1,\dots,n\} \, \text{ and for all } i\in \N\, .
\end{equation}
Then
\begin{equation*}
\mu\left( \limsup_{i \to \infty} R(\bx_{i}, \br_{i}) \right)=\mu\left(\limsup_{i \to \infty} U_{i} \right)\,.
\end{equation*}
\end{lemma}

\begin{proof}
For ease of notation suppose we are in a two-dimensional product space ($n=2$). It will become clear that the argument generalises to $n$-dimensional product space. 
Let us begin by showing that
\begin{equation*}
    \mu\left(\limsup_{i\to \infty} R(\bx_i,\br_i) \right)=\mu\left(\limsup_{i\to \infty} U_{i}^{(1)}\times B^{(2)}(x_{i,2},r_{i,2})\right)\, .
\end{equation*}
For any $y \in X_{2}$ let
\begin{equation*}
I_{y}= \left\{ i\in\N: \, \exists \, \, x \in X_{1} \;\;\text{such that }\, \, (x,y) \in R(\bx_i,br_i) \right\},
\end{equation*}
and for any subset $F \subseteq X$ let $F_y$ denote the fiber of $F$ at $y$, that is
\begin{equation*}
F_y=\{ x: (x,y) \in F\} \subseteq X_{1}.
\end{equation*}
Observe that
\begin{equation}\label{same1}
A_y:=\left( \limsup_{i \to \infty} \, R(\bx_i,br_i) \right)_{y}=\underset{i \in I_y}{\limsup_{i \to \infty}}\, R(\bx_i,br_i)_{y}=\underset{i \in I_y}{\limsup_{i \to \infty}}\, B^{(1)}(x_{i,1},r_{i,1}) =:B(y).
\end{equation}
Indeed, if $x \in A_y$ then it implies there exists an infinite sequence $\{i_{k}\}$ such that
\begin{equation*}
(x,y) \in R(\bx_{i_k},\br_{i_k}) \quad \text{ for all } i_k.
\end{equation*}
Hence $\{i_k\} \subseteq I_y$ and so $x \in B(y)$.

Conversely, if $x \in B(y)$ then $B(y)$ is non-empty and so $I_y$ must be infinite. By the definition of $I_y$ and the fact that $x \in B(y)$, we have that $x \in R(\bx_i,\br_i)_{y}$ for infinitely many $i \in I_y$, and so $x \in A_y$. \par

Similarly, we have that
\begin{equation}\label{same2}
C_y:= \left( \limsup_{i \to \infty} U_{i}^{(1)}\times B^{(2)}(x_{i,2},r_{i,2}) \right)_{y}=\underset{i \in I_y}{\limsup_{i \to \infty}} \left(U_{i}^{(1)}\times B^{(2)}(x_{i,2},r_{i,2})\right)_{y}=\underset{i \in I_y}{\limsup_{i \to \infty}} U_{i}^{(1)} =: D(y).
\end{equation}

By Lemma~\ref{CI_balls} we get that
\begin{equation} \label{same3}
\mu_{1}(B(y))=\mu_{1}(D(y))\, .
\end{equation}

Applying Fubini's Theorem we have that
\begin{align*}
\mu \left( \limsup_{i \to \infty} R(\bx_{i},\delta_{i}) \right) & = \int_{X_{2}} \mu_{1} \left( A_y \right) d \mu_{2} \stackrel{\eqref{same1}}{=} \int_{X_{2}} \mu_{1} \left( B(y) \right) d \mu_{2} \overset{\eqref{same3}}{=} \int_{X_{2}} \mu_{1} \left( D(y)  \right) d \mu_{2}  \stackrel{\eqref{same2}}{=} \int_{X_{2}} \mu_{1} \left( C_y \right) d \mu_{2} \\
&=\mu\left(\limsup_{i\to \infty} U_{i}^{(1)}\times B^{(2)}(x_{i,2},r_{i,2})\right)\,.
\end{align*}
Clearly, the above method can be repeated in the second product space to obtain the desired result.
\end{proof}

Armed with the mass transference principle from rectangles to rectangles and Lemma~\ref{CI_product} we can prove the following:

\begin{theorem} \label{shifted mtprr}
Let $(X,d,\mu)$ be as above and let $(r_{i})_{i\geq 1}$ be a sequence of positive real numbers with $r_{i}\to 0$ as $i\to \infty$ and $(\bx_{i})_{i\geq 1}$ a countable collection of points $\bx_{i}=(x_{i}^{(1)},\dots, x_{i}^{(n)}) \in X$. Let $\bfa=(a_{1},\dots,a_{n})\in\R^{n}_{+}$ and suppose that
\begin{equation} \label{full measure rectangle}
    \mu\left( \limsup_{i\to \infty} R(\bx_{i},r_{i}^{\bfa}) \right)=\mu(X)\, .
\end{equation}
Let $(\mathbf{\gamma_{i}})_{i\geq 1}$ be a sequence of $n$-tuples $\mathbf{\gamma_{i}}=(\gamma_{i}^{(1)},\dots, \gamma_{i}^{(n)})$ such that $x_{i}^{(j)}+\gamma_{i}^{(j)} \in X_{j}$ for each $j=1,\dots,n$, and 
\begin{equation} \label{shift rectangle}
    \limsup_{i\to \infty}\, d_{j}\left(0,\gamma_{i}^{(j)}\right)\, r_{i}^{-a_{j}}< \infty \quad \text{ for } \, j=1,\dots, n \, .
\end{equation}
Then, for any $\bft=(t_{1},\dots, t_{n})\in\R^{n}_{+}$,
\begin{equation*}
\dimh \limsup_{i\to \infty}R\left(\bx_{i}+\mathbf{\gamma_{i}},r_{i}^{\bfa+\bft}\right) \geq \min_{A_{i} \in A} \left\{ \sum\limits_{j \in \cK_{1}} \delta_{j}+ \sum\limits_{j \in \cK_{2}} \delta_{j}+ \kappa \sum\limits_{j \in \cK_{3}} \delta_{j}+(1-\kappa) \frac{\sum\limits_{j \in \cK_{3}}a_{j}\delta_{j}-\sum\limits_{j \in \cK_{2}}t_{j}\delta_{j}}{A_{i}} \right\},
\end{equation*}
where $A=\{ a_{i}+t_{i} , 1 \leq i \leq n \}$ and $\cK_{1},\cK_{2},\cK_{3}$ are a partition of $\{1, \dots, n\}$ defined as in Theorem~\ref{MTPRR}.

\end{theorem}

\begin{remark}\rm
 In the balls to rectangles case of the above theorem i.e. when $\bfa=(a,\dots,a)$, our result can essentially be deduced from \cite[Theorem 3.1]{KR21}, since one can take the sequence of sets $(E_{i})$ appearing in their theorem to be rectangles with centre not necessarily equal to the center of $B_{i}$. Our result is a generalisation since we can go from rectangles to rectangles. Furthermore Theorem~\ref{new} does not follow from \cite{KR21} since they only provide a lower bound dimension result rather than the Hausdorff measure statement.
\end{remark}

The proof of Theorem~\ref{shifted mtprr} is essentially repeated verbatim from the proof of Theorem~\ref{new} with Lemma~\ref{CI_product} in place of Lemma~\ref{CI_balls} and Theorem~\ref{MTPRR} in place of Theorem~\ref{MTP}. We provide the proof for completeness.

\begin{proof}[Proof of Theorem~\ref{shifted mtprr}]
    By \eqref{shift rectangle} there exists a constant $k>0$ such that $d_{j}(0,\gamma_{i}^{(j)})<kr_{i}^{\tfrac{a_{j}s}{\delta}}$ for each $j=1\dots, n$ and for all $i\geq i_{0}$ where $i_{0}$ is chosen sufficiently large. Hence, by applying the triangle inequality to each coordinate of the product space we get that
    \begin{equation*}
        (k+1)R(\bx_{i},r_{i}^{\bfa}) \supseteq R(\bx_{i}+\mathbf{\gamma_{i}},r_{i}^{\bfa})
    \end{equation*}
    Since the measures $\mu_{j}$ are Ahlfors regular and $\bx_{i}+\mathbf{\gamma_{i}} \in X$ for each $i\geq 1$ we have that
    \begin{equation*}
        \mu_{j}\left((k+1)B^{(j)}(x_{i}^{(j)},r_{i}^{a_{j}})\right) \asymp  \mu_{j}\left(B^{(j)}(x_{i}^{(j)},r_{i}^{a_{j}})\right)\asymp \mu_{j}\left(B^{(j)}(x_{i}^{(j)}+\gamma_{i}^{(j)},r_{i}^{a_{j}})\right)\, ,
    \end{equation*}
    where the implied constants are independent of $i$. Thus condition \eqref{CI_product condition} of Lemma~\ref{CI_product} is satisfied and so by Lemma~\ref{CI_product} and \eqref{full measure rectangle} we have that $\mu\left(R(\bx_{i}+\mathbf{\gamma_{i}},r_{i}^{\bfa})\right)=\mu(X)$.
    Hence Theorem~\ref{MTPRR} is applicable and the dimension result follows.
\end{proof}

%
%
%
%


\section{Proof of Theorem~\ref{dim >2 generalised}}\label{main}

We begin with describing the following suitable subsets of $W_{n}(\Psi,\Phi)$ which brings into play the relationship between Diophantine approximation by annuli and the shifted mass transference principles. \par 
For each $j\in \{1,\dots,n\}$ and $(\bp,q)\in \Z^{n}\times\N$ define the hyperrectangles
\begin{equation*}
    R^{j,+}_{\bp,q}(\Psi,\Phi):=\left(\prod_{i=1}^{j-1}B\left(\frac{p_{i}}{q},\frac{\psi_{i}(q)}{q}\right) \right)\times B\left(\frac{p_{j}}{q}+\frac{(2-\phi_{j}(q))\psi_{j}(q)}{2q}, \frac{\phi_{j}(q)\psi_{j}(q)}{2q}\right)\times \left(\prod_{i=j+1}^{n}B\left(\frac{p_{i}}{q},\frac{\psi_{i}(q)}{q}\right) \right)\, ,
\end{equation*}
and
\begin{equation*}
    R^{j,-}_{\bp,q}(\Psi,\Phi):=\left(\prod_{i=1}^{j-1}B\left(\frac{p_{i}}{q},\frac{\psi_{i}(q)}{q}\right) \right)\times B\left(\frac{p_{j}}{q}-\frac{(2-\phi_{j}(q))\psi_{j}(q)}{2q}, \frac{\phi_{j}(q)\psi_{j}(q)}{2q}\right)\times \left(\prod_{i=j+1}^{n}B\left(\frac{p_{i}}{q},\frac{\psi_{i}(q)}{q}\right) \right)\, .
\end{equation*}
Let
\begin{equation*}
    R_{q}^{j,+}(\psi,\phi)=  \bigcup_{0\leq \|\bp\|\leq q}  R_{\bp,q}^{j,+}(\psi,\phi) \quad \text{ and similarly } \quad  R_{q}^{j,-}(\psi,\phi)=  \bigcup_{0\leq \|\bp\|\leq q}  R_{\bp,q}^{j,-}(\psi,\phi)\, .
\end{equation*}
Then define
\begin{equation*}
    \widetilde{W}^{j}_{n}(\Psi,\Phi):= \limsup_{q\to \infty} \, R_{q}^{j,+}(\Psi,\Phi)\, .
\end{equation*}
Notice $\widetilde{W}^{j}_{n}(\Psi,\Phi)$ can be seen as a $\limsup$ set of rectangles with their centres perturbed by $\tfrac{(2-\phi_{j}(q))\psi_{j}(q)}{2q}$ in the $j$th coordinate axis. We have the following lemma which essentially follows by using the triangle inequality.

\begin{lemma}\label{lebesgue subset} 
For each $j \in \{1,\dots,n\}$ we have that $\widetilde{W_{n}}^{j}(\Psi,\Phi)\subseteq W_{n}(\Psi,\Phi)$.
\end{lemma}

\begin{proof}
    It is clear that in the $n-1$ coordinates of the product space, excluding $j$, the hyperrectangle $R^{j,+}_{\bp,q}(\Psi,\Phi)$ coincides with the annulus $A_{\bp,q}(\Psi,\Phi)$. To verify the $j$th coordinate axis observe that by the triangle inequality
\begin{align*}
    \left|x_{j}-\frac{p_{j}}{q}\right|&=\left| x_{j}-\left(\frac{p_{j}}{q}+\frac{(2-\phi_{j}(q))\psi_{j}(q)}{2q}\right)+\frac{(2-\phi_{j}(q))\psi_{j}(q)}{2q} \right|\\ &\leq  \frac{\phi_{j}(q)\psi_{j}(q)}{2q}+\frac{(2-\phi_{j}(q))\psi_{j}(q)}{2q}\\ &=\frac{\psi_{j}(q)}{q}\, .
\end{align*}
By the reverse triangle inequality
\begin{align*}
    \left| x_{j}-\frac{p_{j}}{q}\right| &\geq \left|\frac{(2-\phi_{j}(q))\psi_{j}(q)}{2q}\right|- \left| x_{j}-\left(\frac{p_{j}}{q}+\frac{(2-\phi_{j}(q))\psi_{j}(q)}{2q}\right)\right| \\ &\geq \frac{(2-\phi_{j}(q))\psi_{j}(q)}{2q}- \frac{\phi_{j}(q)\psi_{j}(q)}{2q}\\ &=\frac{(1-\phi_{j}(q))\psi_{j}(q)}{q}\, .
\end{align*}
Combining these, and trivially on the remaining $n-1$ coordinate axes, we conclude that if $\bx \in \widetilde{W_{n}}^{j}(\Psi,\Phi)$ then there exists infinitely many rational points $\frac{\bp}{q}$ such that
\begin{equation*}
    \left|x_{i}-\frac{p_{i}}{q}\right|< \frac{\psi_{i}(q)}{q}\, \quad i=1,\dots,n \quad \text{and} \quad \frac{(1-\phi_{j}(q))\psi_{j}(q)}{q}<\left|x_{j}-\frac{p_{j}}{q}\right|< \frac{\psi_{j}(q)}{q}\, \quad \text{ for some } j\in \{1,\dots,n\}\, ,
\end{equation*}
and so $\bx \in W_{n}(\Psi,\Phi)$.
\end{proof}

It should be clear from the above proof that we also have that 
\begin{equation*}
    A_{\bp,q}(\Psi,\Phi)=\bigcup_{j=1}^{n}\left(R^{j,+}_{\bp,q}(\Psi,\Phi)\cup R^{j,-}_{\bp,q}(\Psi,\Phi)\right).
\end{equation*}
This observation will be used when considering the upper bound Hausdorff dimension of Theorem~\ref{dim n>1}\,.

We split this proof into an upper and lower bound case.
\subsection*{Upper bound}
    For the upper bound we will use the previous observation that for each $(\bp,q)\in \Z^{n}$ the rectangular annulus can be sufficiently split, that is
    \begin{equation} \label{rectangular annulus split}
    A_{\bp,q}(\Psi,\Phi)=\bigcup_{j=1}^{n}\left(R^{j,+}_{\bp,q}(\Psi,\Phi)\cup R^{j,-}_{\bp,q}(\Psi,\Phi)\right)\, .
\end{equation}
We will cover each shifted hyperrectangle by balls of various sizes depending on their sidelengths. Recall the function
\begin{equation*}
        \tau_{k(j)}=\begin{cases}
            \tau_{\psi_j}+\tau_{\phi_j} \quad \text{ if } k(j)=j\, ,\\
            \tau_{\psi_{k(j)}}  \quad \qquad \text{ otherwise.}
        \end{cases}
    \end{equation*}
For each $k(j)\in \{1,\dots,n\}$ we will associate a specific covering of balls. Fix $k(j)\in \{1,\dots,n\}$. Using a standard geometric argument for each $(\bp,q)\in\Z^{n}\times \N$ and each $j\in \{1,\dots,n\}$ we may cover $R^{j,+}_{\bp,q}(\Psi,\Phi)$ by a number of balls comparable with 
\begin{equation*}
    \max\left\{1,\frac{q^{-\tau_{\psi_j}-\tau_{\phi_j}}}{q^{-\tau_{k(j)}}}\right\}\prod_{i=1, i\neq j}^{n} \max\left\{1, \tfrac{q^{-\tau_{\psi_i}}}{q^{-\tau_{k(j)}}} \right\}
\end{equation*}
    each of radius $q^{-(1+\tau_{k(j)})}$. The same can be seen for $R^{j,-}_{\bp,q}(\Psi,\Phi)$. Now, repeating the same argument in each $j\in\{1,\dots,n\}$, and noting \eqref{rectangular annulus split} we have that
    \begin{align*}
        \cH^{s}\left(W_{n}(\Psi,\Phi)\right) &\ll \lim_{N\to \infty} \sum_{q\geq N} q^{n} \left(2\sum_{j=1}^{n}\max\left\{1,\frac{q^{-\tau_{\psi_j}-\tau_{\phi_j}}}{q^{-\tau_{k(j)}}}\right\}\prod_{i=1, i\neq j}^{n} \max\left\{1, \tfrac{q^{-\tau_{\psi_i}}}{q^{-\tau_{k(j)}}} \right\} q^{-(1+\tau_{k(j)})s}\right) \\
        & \ll \lim_{N\to \infty} \sum_{j=1}^{n}\left(\sum_{q\geq N} q^{n+ \sum_{i\neq j : \tau_{\psi_i}<\tau_{k(j)}}(\tau_{k(j)}-\tau_{\psi_i})+\max\{0,\tau_{k(j)}-\tau_{\psi_j}-\tau_{\phi_j} \} -(1+\tau_{k(j)})s } \right)=0\, ,
    \end{align*}
    when
    \begin{equation*}
        s>\max_{j=1,\dots ,n}\,\, \min_{k(j)=1,\dots,n} \left\{ \frac{n+1+\sum_{i\neq j : \tau_{\psi_i}<\tau_{k(j)}}(\tau_{k(j)}-\tau_{\psi_i})+\max\{0,\tau_{k(j)}-\tau_{\psi_j}-\tau_{\phi_j}\}}{1+\tau_{k(j)}} \right\} \, .
    \end{equation*}
    The minimum over the choices of $k(j)\in \{1,\dots,n\}$ appears due to our choice to optimize the cover of $R_{q}^{j,+}(\Psi,\Phi)\cup R_{q}^{j,-}(\Psi,\Phi)$. That is, choose the cover with minimum Hausdorff $s$-cost. The maximum is taken to ensure that when we consider the collective cover by combining the covers for each $j$th coordinate product space the summations in the above Hausdorff $s$-measure calculation are convergent for each $j=1,\dots,n$, and so as $N\to \infty$ the total of the sums tends to zero. This completes the upper bound.

\subsection*{Lower bound} The lower bound is an application of Theorem~\ref{shifted mtprr}. For now consider the $\limsup$ set of hyperrectangles $\widetilde{W}_{n}^{j}(\Psi,\Phi)$. Firstly, by Minkowksi's Theorem for systems of linear forms, for any $n$-tuple $(b_{1},\dots,b_{n})\in\R^{n}_{+}$ satisfying $\sum_{i=1}^{n}b_{i}=1$, and $\Gamma(q)=(q^{-b_{1}},\dots, q^{-b_{n}})$, we have that
\begin{equation}\label{minkowksi full measure consequence}
\lambda_{n}(W_{n}(\Gamma))=\lambda_{n}([0,1]^{n})\, .    
\end{equation}
Set the $n$-tuple $\bfa$ appearing in Theorem~\ref{shifted mtprr} to be
\begin{equation*}
    a_{1}=1+b_{1}\, , \, \, a_{2}=1+b_{2}\, , \, \dots \, \, a_{n}=1+b_{n}\, . 
\end{equation*}
Without loss of generality suppose the exponents $\tau_{\psi_{1}},\dots, \tau_{\psi_{n}}$ are ordered so that $\tau_{\psi_{1}}\geq \dots \geq \tau_{\psi_{n}}>0$. Consider two cases.
\begin{itemize}
    \item $\tau_{\psi_i}\geq \tfrac{1}{n}$ for all $i=1,\dots,n$: Then pick $b_{1}=\dots=b_{n}=\tfrac{1}{n}$, and set
    \begin{equation*}
        a_{i}=1+\tfrac{1}{n}\, , \, \text{ and } \, t_{i}=\tau_{\psi_i}-\tfrac{1}{n}\geq 0 \, \quad i \in \{1,\dots, n\}\backslash \{j\}\, .
    \end{equation*}
    For the $j$th product space set 
    \begin{equation*}
        b_{j}=\tfrac{1}{n} \, , \, \text{ and } \, t_{j}=\tau_{\psi_j}+\tau_{\phi_j}-\tfrac{1}{n}\geq 0 \, .
    \end{equation*}
    Note that our condition $\sum_{i}b_{i}=1$ is satisfied so by Minkowski's Theorem for systems of linear forms we have \eqref{minkowksi full measure consequence} and so \eqref{full measure rectangle} is satisfied. Taking the sequence of $n$-tuples
\begin{equation}\label{gamma sequence}
    \gamma_{j}(q)=\tfrac{(2-\phi_{j}(q))\psi_{j}(q)}{2q}<\frac{\psi_{j}(q)}{q}, \quad \text{ and } \quad \gamma_{i}(q)=0 \, \quad i \in \{1,\dots,n\}\backslash\{j\}
\end{equation}
we clearly have that 
\begin{equation*}
\lim_{q\to \infty}\frac{|\gamma_{j}(q)|}{q^{-1-\tfrac{1}{n}}}< \infty\, ,
\end{equation*}
and so condition~\ref{shift rectangle} is satisfied. Thus Theorem~\ref{shifted mtprr} is applicable. Directly substituting our values for $a_{i}$ and $t_{i}$ into the formula gives us the lower bound Hausdorff dimension when $j$ is fixed. That is, we have
\begin{equation*}
    \dimh \widetilde{W}_{n}^{j}(\Psi,\Phi) \geq \min_{k(j)=1,\dots,n} \left\{ \frac{n+1+\sum_{i\neq j : \tau_{\psi_i}<\tau_{k(j)}}(\tau_{k(j)}-\tau_{\psi_i})+\max\{0,\tau_{k(j)}-\tau_{\psi_j}-\tau_{\phi_j}\}}{1+\tau_{k(j)}} \right\}\, .
\end{equation*}
We should remark at this stage that the case of $\tau_{\psi_{i}}\geq \tfrac{1}{n}$ means we can use a `balls to rectangles' mass transference principle. So in particular the dimension result appearing here also follows from \cite[Theorem 3.1]{KR21}. This is not true in the following case.

\item There exists $1\leq k \leq n$ such that $\tau_{\psi_k}<\tfrac{1}{n}$: Let $1\leq \ell \leq n-1$ be the largest integer such that 
\begin{equation*}
    \tau_{\psi_\ell}>\frac{1-\sum_{i=\ell+1}^{n}\tau_{\psi_i}}{\ell}\, .
\end{equation*}
Set
\begin{equation*}
    b_{i}=\tau_{\psi_i} \quad \text{ for } i=\ell+1 ,\dots, n\, , \text{ and } b_{i}=\frac{1-\sum_{i=\ell+1}^{n}\tau_{\psi_i}}{\ell} \quad \text{ for } i=1,\dots,\ell \, ,
\end{equation*}
and let $t_{i}=\tau_{\psi_{i}}-b_{i}\geq 0$ for $i\in \{1,\dots,n\}\backslash \{j\}$ and $t_{j}=\tau_{\psi_j}+\tau_{\phi_j}-b_{j}\geq 0$ .
Notice that again \eqref{minkowksi full measure consequence} is satisfied, and so \eqref{full measure rectangle} is satisfied. Take the same sequence of $n$-tuples as \eqref{gamma sequence} and note that, independent of whether $j\leq \ell$ or $j>\ell$, we have that
\begin{equation*}
    \lim_{q\to \infty}\frac{|\gamma_{j}(q)|}{q^{-1-b_j}}< \infty \, ,
\end{equation*}
so \eqref{shift rectangle} is satisfied. Applying Theorem~\ref{shifted mtprr} and substituting in our choices of $a_{i}$ and $t_{i}$ we obtain the same conclusion as in the previous case, that is 
\begin{equation*}
    \dimh \widetilde{W}_{n}^{j}(\Psi,\Phi) \geq \min_{k(j)=1,\dots,n} \left\{ \frac{n+1+\sum_{i\neq j : \tau_{\psi_i}<\tau_{k(j)}}(\tau_{k(j)}-\tau_{\psi_i})+\max\{0,\tau_{k(j)}-\tau_{\psi_j}-\tau_{\phi_j}\}}{1+\tau_{k(j)}} \right\}\, .
\end{equation*}
\end{itemize}

To complete the lower bound proof notice we can apply the same method to each set $\widetilde{W}_{n}^{j}(\Psi,\Phi)$ for $j=1,\dots, n$ to obtain a lower bound Hausdorff dimension result. Taking the maximum dimension over these gives us the lower bound result thus completing the proof.

\noindent{\bf Acknowledgments.} This research is supported by the Australian
Research Council Discovery Project (200100994).

\end{document}